\documentclass[11pt,a4paper]{article}

\usepackage[T1]{fontenc}
\usepackage[english]{babel}
\usepackage{lmodern} % robust fonts on arXiv

\usepackage{amsmath,amssymb,amsfonts,amsthm}
\usepackage{mathtools}
\usepackage{bm}
\usepackage{mathrsfs}

\usepackage{graphicx}
\usepackage{subcaption}
\usepackage{booktabs}
\usepackage{array}
\usepackage{float} % allows [H] (keep if you really want it)
\usepackage{tikz}  % if TikZ causes issues on arXiv, comment this line

\usepackage{microtype}
\usepackage{enumitem}
\usepackage{xcolor}
\usepackage[colorlinks=true,linkcolor=blue,citecolor=blue,urlcolor=blue]{hyperref}

\usepackage{geometry}
\theoremstyle{plain}
\newtheorem{theorem}{Theorem}[section]
\newtheorem{proposition}[theorem]{Proposition}
\newtheorem{lemma}[theorem]{Lemma}

\theoremstyle{definition}
\newtheorem{definition}[theorem]{Definition}
\theoremstyle{remark}
\newtheorem{remark}[theorem]{Remark}

\numberwithin{equation}{section}
\allowdisplaybreaks

\DeclareMathOperator{\dist}{dist}

\newcommand{\C}{\mathbb{C}}

\newcommand{\Lap}{\mathcal{L}}

\newcommand{\Dom}{D}
\newcommand{\Res}{\rho}

\newcommand{\Sig}{\Sigma}
\newcommand{\Gtheta}{\Gamma_\theta}

\newcommand{\Kgen}{K}
\newcommand{\Vgen}{V}

\newcommand{\Kab}{K_{\alpha}^{\mathrm{ABC}}}
\newcommand{\Phiab}{\Phi_{\alpha}}

\newcommand{\WD}{\,^{W}\!D}
\newcommand{\Kw}{K_{\alpha,\beta}^{W}}
\newcommand{\Psym}{\Psi_{\alpha,\beta}}
\newcommand{\Bnorm}{B_{\alpha,\beta}}

\title{\textbf{Rational-Kernel Fractional Evolution Equations with Almost Sectorial Operators: A Resolvent Framework Unifying ABC and W Dynamics}}
\author{Mohamed Wakrim\\
\small Ibn Zohr University, Faculty of Sciences, Agadir, Morocco\\
\small \texttt{m.wakrim@uiz.ac.ma}}
\date{\today}

\newcommand{\keywords}[1]{%
  \par\smallskip
  \noindent\textbf{Keywords.} #1\par
}
\newcommand{\msc}[1]{%
  \par\smallskip
  \noindent\textbf{MSC 2020.} #1\par
}

\begin{document}
\maketitle

\begin{abstract}
Fractional evolution equations are widely used to model diffusion processes with memory effects and anomalous time scaling.
Most existing theories rely on Caputo-type derivatives combined with sectorial operators. In many applied diffusion models, including Kimura and Bessel-type operators, the natural generator is only almost sectorial: its resolvent is well controlled at high frequencies but may fail to remain bounded near the origin. This prevents a direct use of standard Caputo-based frameworks.

\noindent In this work, we develop a resolvent approach for fractional evolution equations with rational, nonsingular kernels, including the Atangana--Baleanu (ABC) derivative and a two-parameter W-derivative. The analysis is based on the Laplace-domain mapping $z = s^{\alpha-1}$, which shifts the inversion contour into the high-frequency resolvent regime of the generator; see Lemma~\ref{lem:geometry}. Within this framework, we prove well-posedness, strong continuity of the associated resolvent families, and fractional smoothing estimates of order $t^{-\alpha\gamma}$.
\end{abstract}


\msc{47D06; 34A08; 35K65; 35R11; 35B65; 47A10}

% ==========================================================
\section{Introduction}\label{sec:introduction}
% ==========================================================

Fractional evolution equations have become a central tool for modeling diffusion processes with memory, hereditary effects, and anomalous time scaling. In the classical setting, fractional time derivatives---most notably the Caputo derivative---are combined with sectorial operators generating analytic semigroups. This approach yields a well-developed theory encompassing existence, uniqueness, regularity, and long-time behavior of solutions; see, for instance, \cite{Podlubny1999,Pruss1993,Haase2006}.
For resolvent-based perspectives on fractional evolution equations beyond the classical analytic semigroup framework, we also refer to
\cite{Bazhilov2020,LiZheng2015}.

In Section~\ref{sec:caputo-failure}, we show rigorously that the classical Caputo fractional dynamics is structurally incompatible with almost sectorial operators. The obstruction is not technical but geometric: the Laplace symbol $s^\alpha$ forces resolvent evaluation near the spectral origin, where almost sectoriality
provides no control.

\medskip

In a broad range of applied models, the natural spatial generator fails to be sectorial. Typical examples include degenerate or singular elliptic operators, where ellipticity deteriorates near parts of the boundary or at singular points of the
domain. Kimura-type operators in population genetics and Bessel-type operators associated with radial diffusions are representative cases. In appropriate weighted $L^2$ spaces, these operators are often self-adjoint and nonpositive, yet their resolvent satisfies estimates only for large spectral parameters.
Near the origin, the resolvent may blow up or fail to be uniformly controlled.
Such operators belong to the class of \emph{almost sectorial} operators.

\medskip
\noindent\textbf{A structural obstruction (Caputo).}
Consider the Caputo fractional evolution equation
\begin{equation}\label{eq:intro-caputo}
{}^{\mathrm C}D_t^\alpha u(t) + Au(t) = f(t),
\qquad 0<\alpha<1,
\end{equation}
where $A$ is almost sectorial on a Banach space $X$.
In the Laplace domain, \eqref{eq:intro-caputo} leads to resolvent expressions of the form $(s^\alpha I + A)^{-1}$. During Laplace inversion, the contour necessarily approaches the origin $s=0$ and, since $s^\alpha \to 0$ as $s \to 0$, the resolvent is forced to probe the spectral region near $z=0$. This is precisely where almost sectoriality provides no estimate.
The issue is intrinsic: the Caputo symbol $s^\alpha$ drives the analysis into the low-frequency regime. Section~\ref{sec:caputo-failure} makes this obstruction precise and explains why Caputo dynamics generally fails for almost sectorial generators.

\medskip
\noindent\textbf{Redirection via rational nonsingular kernels (ABC, W).}
A different situation arises for fractional derivatives with rational, nonsingular kernels, such as the Atangana--Baleanu derivative (ABC) and more general two-parameter W derivatives.
Although ABC and W originate from different constructions, they share a key Laplace-domain feature: the resolvent is evaluated at an argument involving $s^{\alpha-1}$ rather than $s^\alpha$. Since $\alpha-1<0$, low Laplace frequencies $|s| \to 0$ are mapped to large spectral
parameters $|s^{\alpha-1}| \to \infty$. This mechanism replaces the problematic evaluation of $(zI-A)^{-1}$ near $z=0$ by resolvent estimates at large spectral parameters, precisely where almost sectorial operators admit uniform control.

\medskip
\noindent\textbf{Aim and contributions.}
Our aim is to build a unified resolvent framework that makes the above redirection mechanism transparent. The starting point is an admissible Laplace multiplier $\Kgen(s)$ and the associated contour-integral resolvent family. We first prove that this family is well defined: the contour integral converges absolutely in operator norm and yields bounded operators on $X$. We then show that the family depends continuously on time and satisfies a variation-of-constants formula, which provides the natural notion of mild solution. A key outcome is a fractional smoothing estimate: for every $\gamma\in[0,1)$, the action of $A^\gamma$ along the flow obeys the decay scale $t^{-\alpha\gamma}$. After establishing the abstract results, we verify that the ABC and W kernels fit into this admissible class, so that both dynamics are covered by the same mechanism. Finally, we illustrate the scope of the theory on two representative degenerate generators, namely Kimura-type and Bessel-type diffusion operators, and we also prove that in this setting
the mild and weak formulations are equivalent.

\medskip
\noindent\textbf{Geometric mechanism.}
For clarity, Figure~\ref{fig:geometry} sketches the contour mapping in the complex plane: the left-sectorial Laplace contour $\Gtheta$ is mapped by $s \mapsto z = s^{\alpha-1}$ into a region avoiding the forbidden sector of $A$, while small values of $|s|$ correspond to large values of $|z|$.

\begin{figure}[H]
\centering
\begin{subfigure}[b]{0.48\textwidth}
\centering
\begin{tikzpicture}[scale=1.0]
\draw[->] (-2.3,0) -- (2.3,0) node[right] {$\Re s$};
\draw[->] (0,-2.3) -- (0,2.3) node[above] {$\Im s$};
\draw[thick] (0,0) -- (-1.9,1.3);
\draw[thick] (0,0) -- (-1.9,-1.3);
\node at (-1.55,1.45) {$\Gtheta$};
\node at (-0.55,1.05) {$\theta$};
\end{tikzpicture}
\caption{Left-sectorial contour $\Gtheta$ in the $s$-plane.}
\end{subfigure}\hfill
\begin{subfigure}[b]{0.48\textwidth}
\centering
\begin{tikzpicture}[scale=1.0]
\draw[->] (-2.3,0) -- (2.3,0) node[right] {$\Re z$};
\draw[->] (0,-2.3) -- (0,2.3) node[above] {$\Im z$};
\draw[thick, dashed] (0,0) -- (1.9,1.1);
\draw[thick, dashed] (0,0) -- (1.9,-1.1);
\node at (1.35,1.35) {$\Sig_{\theta_A}$};
\node at (-1.15,1.45) {$z=s^{\alpha-1}$};
\end{tikzpicture}
\caption{Image avoids $\Sig_{\theta_A}$ and probes large $|z|$ when $|s|\to0$.}
\end{subfigure}
\caption{Geometric redirection induced by the mapping $z=s^{\alpha-1}$.}
\label{fig:geometry}
\end{figure}
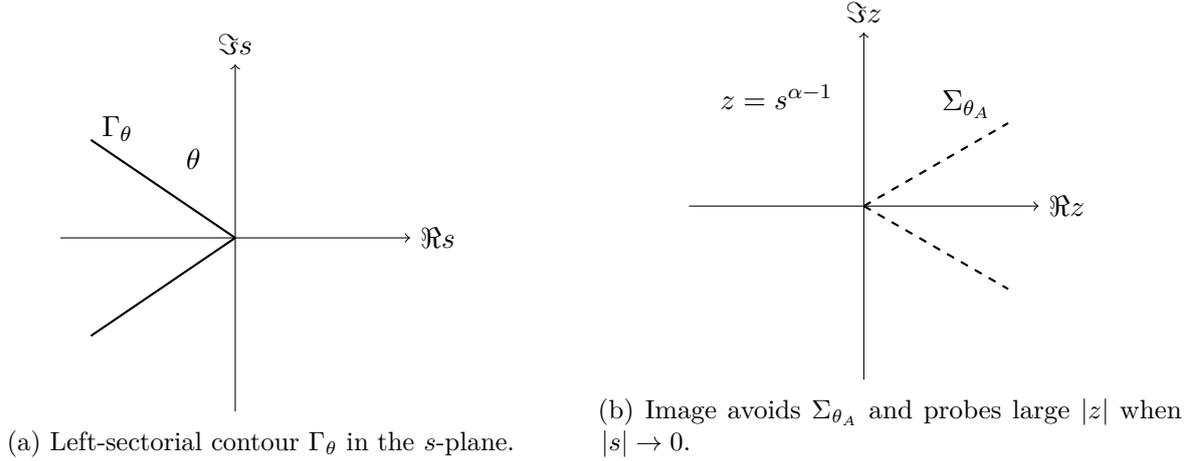

\noindent The two-parameter W-operator considered in this work constitutes a new fractional derivative. A detailed study of its kernel properties will appear in a forthcoming work \cite{WoperatorRef}. To the best of our knowledge, it has not appeared previously in the literature. Its defining feature is that the associated Laplace multiplier preserves the same geometric redirection $z=s^{\alpha-1}$ as the Atangana--Baleanu kernel, while introducing an additional parameter $\beta$ that modulates memory concentration without altering the resolvent geometry nor the smoothing scale.

% ==========================================================
\section{Preliminaries}\label{sec:preliminaries}
% ==========================================================

Throughout this work, $X$ denotes a complex Banach space and
$A:\Dom(A)\subset X\to X$ a closed, densely defined linear operator.
We recall the notion of almost sectorial operators and the high-frequency tools
required later.

% ----------------------------------------------------------
\subsection{Almost sectorial operators}\label{subsec:asectorial}
% ----------------------------------------------------------

\begin{definition}[Almost sectorial operator]\label{def:almost-sectorial}
The operator $A$ is called \emph{almost sectorial} if there exist
$\theta_A\in(0,\pi)$, $M>0$, and $R>0$ such that
\[
\C\setminus \Sig_{\theta_A} \subset \Res(A),\qquad
\Sig_{\theta_A}:=\{z\in\C\setminus\{0\}:|\arg z|<\theta_A\},
\]
and
\begin{equation}\label{eq:asectorial-res}
\|(zI-A)^{-1}\|\le M|z|^{-1},\qquad z\notin\Sig_{\theta_A},\ |z|\ge R.
\end{equation}
\end{definition}

\noindent Almost sectoriality imposes \emph{no uniform resolvent bound} as $|z|\to0$. Almost sectorial operator classes and their functional calculus at high frequencies are discussed in connection with maximal regularity and $H^\infty$-calculus techniques; see, e.g., \cite{KunstmannWeis2004,Haase2006}. In particular, $0$ may belong to $\sigma(A)$ and the resolvent may blow up or even fail to exist near the origin.

\begin{remark}[Interpretation and scope]
Almost sectorial operators naturally arise for degenerate or singular diffusion generators. In weighted spaces they may be self-adjoint and dissipative, with robust high-frequency resolvent structure but poor control near $z=0$.
\end{remark}

% ----------------------------------------------------------
\subsection{Fractional powers and a high-frequency bound}\label{subsec:frac-powers}
% ----------------------------------------------------------

For $\gamma\in[0,1)$, fractional powers $A^\gamma$ can be defined through a functional calculus restricted to contours avoiding the origin (the precise construction depends on the operator class). We assume the availability of the standard high-frequency estimate. Such bounds are standard for sectorial operators and extend to several almost
sectorial settings when the functional calculus is restricted away from the origin; see \cite{Haase2006,KunstmannWeis2004}.

\begin{equation}\label{eq:fracpower-res}
\|A^\gamma (zI-A)^{-1}\|\le C_\gamma |z|^{\gamma-1},\qquad
z\notin\Sig_{\theta_A},\ |z|\ge R,
\end{equation}
for some $C_\gamma>0$ and all $\gamma\in[0,1)$.

\begin{remark}
Estimate \eqref{eq:fracpower-res} is classical for sectorial operators and remains valid for almost sectorial operators under the sole assumption of high-frequency resolvent control, i.e.,
$z \in \mathbb{C}\setminus\Sigma_{\theta}$ with $|z|$ sufficiently large. No resolvent bounds near the origin are required in the present framework.
\end{remark}

% ==========================================================
\section{Failure of Caputo fractional dynamics}\label{sec:caputo-failure}
% ==========================================================

We recall why Caputo fractional equations are, in general, incompatible with almost sectorial generators without additional low-frequency assumptions.

\noindent Consider the homogeneous Caputo problem
\begin{equation}\label{eq:caputo-eq}
{}^{\mathrm{C}}D_t^\alpha u(t) + Au(t) = 0,
\qquad 0<\alpha<1,\quad u(0)=u_0.
\end{equation}
Taking Laplace transforms yields
\begin{equation}\label{eq:caputo-resolvent}
\widehat{u}(s)
=
(s^\alpha I + A)^{-1}\, s^{\alpha-1}u_0,
\qquad \Re s>0.
\end{equation}

\begin{proposition}[Structural failure of Caputo dynamics]\label{prop:caputo-failure}
Let $A$ be almost sectorial. Assume that either $0\in\sigma(A)$ or that $\|(zI-A)^{-1}\|$ is not uniformly bounded as $z\to0$. Then, in general, the Laplace inversion of \eqref{eq:caputo-resolvent} cannot be justified using only the high-frequency resolvent estimate \eqref{eq:asectorial-res}.
\end{proposition}

\begin{proof}
Formally, Laplace inversion yields
\[
u(t)=\frac{1}{2\pi i}\int_{\Gamma} e^{st}\,(s^\alpha I+A)^{-1}s^{\alpha-1}u_0\,ds,
\]
where $\Gamma$ is a Bromwich-type contour and therefore approaches $s=0$. Since $0<\alpha<1$, we have $s^\alpha\to0$ as $s\to0$ along $\Gamma$, so the resolvent factor $(s^\alpha I+A)^{-1}$ necessarily probes spectral parameters $z=s^\alpha$ arbitrarily close to $0$.

\noindent Almost sectoriality provides the estimate
$\|(zI-A)^{-1}\|\le M|z|^{-1}$ only for $|z|\ge R$, and thus offers no control in a neighborhood
of the origin.
If $0\in\sigma(A)$, the resolvent is singular at $z=0$.
If $0\notin\sigma(A)$ but $\|(zI-A)^{-1}\|$ is unbounded as $z\to0$, then for $s$ near $0$
the integrand behaves like
\[
\|(s^\alpha I+A)^{-1}s^{\alpha-1}\| \to \infty,
\]
and no integrable majorant can be constructed in a neighborhood of $s=0$.

\noindent In either case, the absolute convergence required to justify Laplace inversion cannot be obtained from the high-frequency estimate \eqref{eq:asectorial-res} alone.
\end{proof}

\begin{remark}
The obstruction is geometric: Caputo dynamics forces the resolvent to be evaluated
near the origin in the spectral variable $z=s^\alpha$.
\end{remark}

% ==========================================================
\section{An abstract kernel resolvent framework}\label{sec:abstract-K}
% ==========================================================

We now isolate the mechanism allowing well-posedness under \emph{high-frequency}
resolvent control only. The key is to replace the Caputo symbol $s^\alpha$ by a kernel
multiplier whose inversion evaluates the resolvent at $z=s^{\alpha-1}$.

% ----------------------------------------------------------
\subsection{Laplace inversion contour}\label{subsec:contour}
% ----------------------------------------------------------

Let $A$ be almost sectorial with angle $\theta_A\in(0,\pi)$.
Fix $\theta\in(\pi/2,\theta_A)$ and define the left-sectorial contour
\[
\Gtheta=\Gtheta^+\cup \Gtheta^-,
\qquad
\Gtheta^\pm = \{re^{\pm i\theta}: r\in(0,\infty)\},
\]
oriented from $+\infty e^{i\theta}$ to $+\infty e^{-i\theta}$.
Complex powers are taken with the principal branch on $\C\setminus(-\infty,0]$.

% ----------------------------------------------------------
\subsection{Admissible kernel multipliers}\label{subsec:admissible}
% ----------------------------------------------------------

\begin{definition}[Admissible kernel multiplier]\label{def:admissible-K}
A function $\Kgen:\C\setminus(-\infty,0]\to\C$ is called \emph{admissible} if it is analytic and
there exist constants $C_0,C_\infty>0$ such that, for all $s\in\Gtheta$,
\begin{equation}\label{eq:K-bounds-abstract}
|\Kgen(s)|\le
\begin{cases}
C_0, & 0<|s|\le 1,\\[1mm]
C_\infty\,|s|^{\alpha-1}, & |s|\ge 1.
\end{cases}
\end{equation}
\end{definition}

\noindent The growth condition $|K(s)| \le C_\infty |s|^{\alpha-1}$ for $|s|\ge 1$ is not claimed to be optimal. It is, however, sufficient to ensure absolute convergence of the
Laplace inversion integral and compatibility with almost sectorial
resolvent bounds. Sharper conditions could be considered, but would not enlarge the class of generators covered by the present theory.

\begin{remark}
Condition \eqref{eq:K-bounds-abstract} is minimal: it ensures integrability near $s=0$ and, combined with $e^{st}$ on $\Gtheta$, integrability at infinity.
\end{remark}

% ----------------------------------------------------------
\subsection{Geometric redirection induced by $z=s^{\alpha-1}$}\label{subsec:geometry}
% ----------------------------------------------------------

\begin{lemma}[Geometric redirection]\label{lem:geometry}
Let $0<\alpha<1$ and let $s\in\Gamma_\theta$, where
\[
\Gamma_\theta^\pm=\{ r e^{\pm i\theta} : r>0\},\qquad \theta\in\bigl(\tfrac{\pi}{2},\pi\bigr),
\]
and complex powers are taken with the principal branch on $\C\setminus(-\infty,0]$.
Set $z=s^{\alpha-1}$.
Then:
\begin{enumerate}[label=(\roman*),leftmargin=7mm]
\item $|z|=|s|^{\alpha-1}\to\infty$ as $|s|\to0$ along $\Gamma_\theta$,
\item $\arg z = (\alpha-1)\arg s$,
\item if $\; (1-\alpha)\theta \ge \theta_A$, then $z\notin \Sigma_{\theta_A}$ (hence $z\in\rho(A)$).
\end{enumerate}
\end{lemma}

\begin{proof}
Write $s=re^{i\varphi}$ with $r=|s|>0$ and $\varphi=\arg(s)\in(-\pi,\pi)$.
Along $\Gamma_\theta$ we have $\varphi=+\theta$ on $\Gamma_\theta^+$ and $\varphi=-\theta$ on
$\Gamma_\theta^-$.

\smallskip
\noindent\emph{(i)} Since $0<\alpha<1$, we have $\alpha-1<0$. Therefore
\[
|z| = |s^{\alpha-1}| = |s|^{\alpha-1} = r^{\alpha-1}.
\]
When $r\to0^+$, the exponent $\alpha-1<0$ forces $r^{\alpha-1}\to+\infty$. This proves (i).

\smallskip
\noindent\emph{(ii)} With the principal branch, complex powers satisfy
\[
s^{\alpha-1} = r^{\alpha-1} e^{i(\alpha-1)\varphi}.
\]
Hence $\arg(z)=\arg(s^{\alpha-1})=(\alpha-1)\varphi=(\alpha-1)\arg(s)$, which is (ii).

\smallskip
\noindent\emph{(iii)} On $\Gamma_\theta^\pm$ we have $\arg(s)=\pm\theta$, so by (ii)
\[
\arg(z)= (\alpha-1)(\pm\theta)=\pm(\alpha-1)\theta.
\]
Taking absolute values and using $\alpha-1<0$ gives
\[
|\arg(z)| = |(\alpha-1)\theta| = (1-\alpha)\theta.
\]
Recall that
\[
\Sigma_{\theta_A}=\{z\in\C\setminus\{0\}:\ |\arg z|<\theta_A\}.
\]
Thus $z\notin\Sigma_{\theta_A}$ is equivalent to $|\arg(z)|\ge\theta_A$.
Under the stated angle condition $(1-\alpha)\theta\ge\theta_A$, we indeed have
\[
|\arg(z)|=(1-\alpha)\theta \ge \theta_A,
\]
and therefore $z\in \C\setminus\Sigma_{\theta_A}\subset\rho(A)$.
This proves (iii).
\end{proof}

% ----------------------------------------------------------
\subsection{High-frequency resolvent bounds along the contour}\label{subsec:HF-bound}
% ----------------------------------------------------------

\begin{lemma}[Resolvent bound along $\Gamma_\theta$]\label{lem:HF-resolvent}
Assume that $A$ is almost sectorial in the sense of Definition~\ref{def:almost-sectorial}, i.e. there exist $\theta_A\in(0,\pi)$, $M>0$ and $R>0$ such that
$\C\setminus\Sigma_{\theta_A}\subset\rho(A)$ and
\[
\|(zI-A)^{-1}\|\le M|z|^{-1}\qquad\text{for all }z\notin\Sigma_{\theta_A},\ |z|\ge R.
\]
Fix $\theta$ and consider $s\in\Gamma_\theta$. Set $z=s^{\alpha-1}$.
Then there exists a constant $C>0$ such that, for every $s\in\Gamma_\theta$,
\begin{equation}\label{eq:resolvent-bound-s}
\|(s^{\alpha-1}I-A)^{-1}\|
\le
\begin{cases}
C\,|s|^{1-\alpha}, & 0<|s|\le 1,\\[1mm]
C, & |s|\ge 1.
\end{cases}
\end{equation}
\end{lemma}

\begin{proof}
Let $s\in\Gamma_\theta$ and set $z=s^{\alpha-1}$. We split the proof into two cases.

\smallskip
\noindent\textbf{Case 1: the region $0<|s|\le 1$.}
Since $0<\alpha<1$, we have $\alpha-1<0$, and therefore
\[
|z| = |s^{\alpha-1}| = |s|^{\alpha-1} = \frac{1}{|s|^{1-\alpha}}.
\]
In particular, when $0<|s|\le 1$, we have $|z|\ge 1$. To use the almost sectorial
estimate, we also need $|z|\ge R$. This is automatic as soon as $|s|$ is small enough:
indeed, $|z|\ge R$ is equivalent to $|s|^{\alpha-1}\ge R$, that is,
\[
|s|\le R^{-1/(1-\alpha)}.
\]
Hence for every $s\in\Gamma_\theta$ with
\[
0<|s|\le \min\{1,\; R^{-1/(1-\alpha)}\},
\]
we have simultaneously $z\notin\Sigma_{\theta_A}$ (by Lemma~\ref{lem:geometry}, with the
appropriate angle condition) and $|z|\ge R$. The resolvent bound
\eqref{eq:asectorial-res} then gives
\[
\|(zI-A)^{-1}\| \le M|z|^{-1} = M\,|s|^{1-\alpha}.
\]

\noindent For the remaining values of $s$ in the annulus
\[
R^{-1/(1-\alpha)} \le |s| \le 1,
\]
the quantity $z=s^{\alpha-1}$ stays in a fixed compact subset of $\rho(A)$:
indeed, $|z|=|s|^{\alpha-1}\in[1,R]$ and $\arg(z)$ is fixed by the contour. Since the resolvent map $z\mapsto (zI-A)^{-1}$ is analytic (hence continuous) on $\rho(A$), it is bounded on that compact set. Therefore there exists $C_1>0$ such that
\[
\|(zI-A)^{-1}\|\le C_1
\qquad\text{whenever }R^{-1/(1-\alpha)} \le |s| \le 1.
\]
Combining the two subcases and increasing constants if necessary, we obtain
\[
\|(s^{\alpha-1}I-A)^{-1}\| \le C\,|s|^{1-\alpha}
\qquad\text{for all }0<|s|\le 1,
\]
because $|s|^{1-\alpha}\ge R^{-1}$ on the annulus, so the uniform bound can be absorbed
into $C|s|^{1-\alpha}$ by enlarging $C$.

\smallskip
\noindent\textbf{Case 2: the region $|s|\ge 1$.}
Consider the image set
\[
\mathcal{K}:=\{\, s^{\alpha-1} : s\in\Gamma_\theta,\ |s|\ge 1 \,\}.
\]
Since the map $s\mapsto s^{\alpha-1}$ is continuous on $\Gamma_\theta$ (with the chosen branch), and $|s|^{\alpha-1}\le 1$ when $|s|\ge 1$, we see that $\mathcal{K}$ is bounded. Moreover, $\mathcal{K}$ is closed (it contains its limit points along the rays), hence $\mathcal{K}$ is compact. By Lemma~\ref{lem:geometry} (again with the appropriate angle condition), $\mathcal{K}\subset\C\setminus\Sigma_{\theta_A}\subset\rho(A)$.
Therefore the resolvent is bounded on $\mathcal{K}$: there exists $C_2>0$ such that
\[
\|(zI-A)^{-1}\|\le C_2 \qquad\text{for all } z\in\mathcal{K}.
\]
Equivalently,
\[
\|(s^{\alpha-1}I-A)^{-1}\|\le C_2 \qquad\text{for all } s\in\Gamma_\theta,\ |s|\ge 1.
\]

\smallskip
\noindent\textbf{Conclusion.}
Taking $C:=\max\{C_1, C_2, M\}$ yields \eqref{eq:resolvent-bound-s}.
\end{proof}

% ----------------------------------------------------------
\subsection{The $\Kgen$-resolvent family}\label{subsec:K-resolvent}
% ----------------------------------------------------------

\begin{definition}[$\Kgen$-resolvent family]\label{def:K-resolvent}
Let $A$ be almost sectorial and let $\Kgen$ be an admissible kernel multiplier. For $t>0$, define
\begin{equation}\label{eq:K-resolvent}
\Vgen_{\Kgen}(t)
:=
\frac{1}{2\pi i}\int_{\Gtheta} e^{st}\,\Kgen(s)\,(s^{\alpha-1}I-A)^{-1}\,ds.
\end{equation}
\end{definition}

\begin{proposition}[Absolute convergence]\label{prop:absolute-convergence}
For each $t>0$, the integral \eqref{eq:K-resolvent} converges absolutely in operator norm and defines a bounded operator on $X$.
\end{proposition}

\begin{proof}
Fix $t>0$. Along $\Gamma_\theta$ we parametrize the two rays by
$s=r e^{\pm i\theta}$, $r\in(0,\infty)$. Since $\theta\in(\pi/2,\pi)$, we have $\Re(s)=r\cos\theta<0$, hence there exists $c:=-\cos\theta>0$ such that
\[
|e^{st}|=e^{t\Re(s)} \le e^{-crt}.
\]
We estimate the integrand in operator norm and split the contour into the cases $0<r\le 1$ and $r\ge 1$.

\smallskip
\noindent\textbf{(i) The case $0<r\le 1$.}
By admissibility \eqref{eq:K-bounds-abstract}, $|\Kgen(s)|\le C_0$ for $|s|\le 1$, and by Lemma~\ref{lem:HF-resolvent} we have
$\|(s^{\alpha-1}I-A)^{-1}\|\le C\,|s|^{1-\alpha}=C\,r^{1-\alpha}$ for $|s|\le 1$. Thus, for $0<r\le 1$,
\[
\bigl\|e^{st}\Kgen(s)(s^{\alpha-1}I-A)^{-1}\bigr\|
\le C_0\,C\, r^{1-\alpha} e^{-crt}.
\]
Since $|ds|=dr$ on each ray, we obtain
\[
\int_{0}^{1} C_0\,C\, r^{1-\alpha} e^{-crt}\,dr <\infty
\qquad\text{because }1-\alpha>-1.
\]

\smallskip
\noindent\textbf{(ii) The case $r\ge 1$.}
By admissibility \eqref{eq:K-bounds-abstract}, $|\Kgen(s)|\le C_\infty |s|^{\alpha-1}
= C_\infty r^{\alpha-1}$ for $|s|\ge 1$, and by Lemma~\ref{lem:HF-resolvent} there is $C'>0$ such that $\|(s^{\alpha-1}I-A)^{-1}\|\le C'$ for $|s|\ge 1$. Hence, for $r\ge 1$,
\[
\bigl\|e^{st}\Kgen(s)(s^{\alpha-1}I-A)^{-1}\bigr\|
\le C_\infty\,C'\, r^{\alpha-1} e^{-crt}.
\]
Thus
\[
\int_{1}^{\infty} C_\infty\,C'\, r^{\alpha-1} e^{-crt}\,dr <\infty,
\]
since the exponential term dominates any polynomial growth.

\smallskip
\noindent Combining (i) and (ii) on both rays shows that the contour integral in \eqref{eq:K-resolvent} is absolutely convergent in operator norm. Moreover, the same estimates give a finite bound on $\|\Vgen_{\Kgen}(t)\|$, so $\Vgen_{\Kgen}(t)$ defines a bounded operator on $X$.
\end{proof}

\begin{lemma}[Laplace transform]\label{lem:Laplace-VK}
For $x\in X$ and $\Re\lambda>0$,
\[
\Lap\{\Vgen_{\Kgen}(\cdot)x\}(\lambda)
=\int_{0}^{\infty}e^{-\lambda t}\Vgen_{\Kgen}(t)x\,dt
=\Kgen(\lambda)\,(\lambda^{\alpha-1}I-A)^{-1}x.
\]
\end{lemma}

\begin{proof}
Fix $x\in X$ and $\Re\lambda>0$. By definition,
\[
\Vgen_{\Kgen}(t)x=\frac{1}{2\pi i}\int_{\Gamma_\theta} e^{st}\Kgen(s)(s^{\alpha-1}I-A)^{-1}x\,ds.
\]
Using the definition \eqref{eq:K-resolvent} and Fubini's theorem (justified by absolute convergence of the contour integral), we write for $\Re\lambda>0$:
\[
\int_0^\infty e^{-\lambda t}\Vgen_{\Kgen}(t)x\,dt
=\frac{1}{2\pi i}\int_{\Gamma_\theta}\left(\int_0^\infty e^{-(\lambda-s)t}\,dt\right)
\Kgen(s)(s^{\alpha-1}I-A)^{-1}x\,ds.
\]
provided we may apply Fubini. We justify this by absolute integrability: along $\Gamma_\theta$ we have $\Re(s)<0$, so $\Re(\lambda-s)\ge \Re\lambda>0$ and
\[
\int_0^\infty |e^{-(\lambda-s)t}|\,dt=\frac{1}{\Re(\lambda-s)}\le \frac{1}{\Re\lambda}.
\]
In addition, Proposition~\ref{prop:absolute-convergence} provides an $L^1$-majorant for
$\|\Kgen(s)(s^{\alpha-1}I-A)^{-1}\|$ on $\Gamma_\theta$. Therefore
\[
\int_{\Gamma_\theta}\int_0^\infty
|e^{-(\lambda-s)t}|\;\|\Kgen(s)(s^{\alpha-1}I-A)^{-1}x\|\,dt\,|ds|<\infty,
\]
and Fubini applies.

\noindent Evaluating the inner integral gives
\[
\int_0^\infty e^{-(\lambda-s)t}\,dt=\frac{1}{\lambda-s},
\qquad \Re\lambda>0,\ \Re s<0.
\]
Thus
\[
\int_0^\infty e^{-\lambda t}\Vgen_{\Kgen}(t)x\,dt
=\frac{1}{2\pi i}\int_{\Gamma_\theta}\frac{\Kgen(s)(s^{\alpha-1}I-A)^{-1}}{\lambda-s}\,ds\;x.
\]
The operator-valued function
$F(s)=\Kgen(s)(s^{\alpha-1}I-A)^{-1}$ is analytic in $s$ on the region enclosed by the contour
and $\lambda$ lies to the right of $\Gamma_\theta$. By the standard Cauchy/contour argument (for the Bromwich-type contour used here), the integral equals $F(\lambda)x$, i.e.
\[
\int_0^\infty e^{-\lambda t}\Vgen_{\Kgen}(t)x\,dt
=\Kgen(\lambda)(\lambda^{\alpha-1}I-A)^{-1}x.
\]
\end{proof}

\begin{theorem}[Strong continuity]\label{thm:strong-continuity}
Let $A$ be almost sectorial and let $\Kgen$ be admissible.
Then $\{\Vgen_{\Kgen}(t)\}_{t>0}$ is strongly continuous on $(0,\infty)$.

\noindent In addition, if in addition $\Vgen_{\Kgen}(t)x \to x$ as $t\downarrow 0$ for every $x\in X$
(which holds for the normalized kernels used in the ABC and W cases),
then the family extends to a strongly continuous family on $[0,\infty)$ by setting
$\Vgen_{\Kgen}(0)=I$.
\end{theorem}

\begin{proof}
\noindent\textbf{Step 1: strong continuity on $(0,\infty)$.}
Fix $x\in X$ and $t_0>0$. Let $t_n\to t_0$ with $t_n>0$. From \eqref{eq:K-resolvent},
\[
\Vgen_{\Kgen}(t_n)x-\Vgen_{\Kgen}(t_0)x
=\frac{1}{2\pi i}\int_{\Gamma_\theta}\bigl(e^{s t_n}-e^{s t_0}\bigr)\,
\Kgen(s)(s^{\alpha-1}I-A)^{-1}x\,ds.
\]
For each fixed $s\in\Gamma_\theta$, $e^{s t_n}\to e^{s t_0}$. To apply dominated convergence, we use that on $\Gamma_\theta$ one has $\Re(s)<0$. For $t$ in a compact interval $[t_0/2,\,2t_0]$,
\[
|e^{st}|\le e^{-c|s|t_0/2},
\]
and the majorants constructed in Proposition~\ref{prop:absolute-convergence} show that
\[
\int_{\Gamma_\theta} e^{-c|s|t_0/2}\,\|\Kgen(s)(s^{\alpha-1}I-A)^{-1}\|\,|ds|<\infty.
\]
As a consequence, the integrand is dominated by an $L^1(\Gamma_\theta)$-function independent of $n$, and dominated convergence yields $\Vgen_{\Kgen}(t_n)x\to \Vgen_{\Kgen}(t_0)x$.

\smallskip
\noindent\textbf{Step 2: extension at $t=0$ under normalization.}
Assume now that $\Vgen_{\Kgen}(t)x\to x$ as $t\downarrow 0$ for every $x\in X$. Then defining $\Vgen_{\Kgen}(0)=I$ gives strong continuity at $t=0$ and hence on $[0,\infty)$.

\noindent In the ABC and W cases this $t\downarrow0$ property follows from the Laplace identity in Lemma~\ref{lem:Laplace-VK} together with a standard Abelian theorem for vector-valued Laplace transforms; see \cite[Ch.~1]{ArendtBattyHieberNeubrander2001}.
\end{proof}

\begin{theorem}[Variation-of-constants]\label{thm:VoC-abstract}
Let $u_0\in X$ and let $f\in L^1_{\mathrm{loc}}([0,\infty);X)$ be such that its Laplace transform $\widehat f(s)$ exists for $\Re s>0$
(e.g.\ $f\in L^1([0,\infty);X)$, or more generally $e^{-\sigma t}f(t)\in L^1([0,\infty);X)$ for every $\sigma>0$). Define for $t\ge0$
\[
u(t)=\Vgen_{\Kgen}(t)u_0+\int_0^t \Vgen_{\Kgen}(t-\tau)\,f(\tau)\,d\tau.
\]
Then, for $\Re s>0$,
\[
\widehat{u}(s)=\Kgen(s)(s^{\alpha-1}I-A)^{-1}u_0
+\Kgen(s)(s^{\alpha-1}I-A)^{-1}\widehat{f}(s).
\]
\end{theorem}

\begin{proof}
Let $t>0$. By definition of the $\Kgen$-resolvent family,
\[
A^\gamma \Vgen_{\Kgen}(t)
=
\frac{1}{2\pi i}
\int_{\Gtheta}
e^{st}\,\Kgen(s)\,A^\gamma(s^{\alpha-1}I-A)^{-1}\,ds.
\]
We first justify the insertion of $A^\gamma$ under the integral.
This will follow from the construction of an integrable majorant in operator norm and from the closedness of the fractional power $A^\gamma$.

\noindent Set $z=s^{\alpha-1}$.  
For $0<|s|\le1$, Lemma~\ref{lem:geometry} implies that $z\in\Res(A)$ and $|z|$ is large. Hence, by the high-frequency resolvent estimate~\eqref{eq:fracpower-res}, there exists a constant $C_\gamma>0$ such that
\[
\|A^\gamma(zI-A)^{-1}\|
\le C_\gamma |z|^{\gamma-1}
= C_\gamma |s|^{(\alpha-1)(\gamma-1)}.
\]
Since $(\alpha-1)(\gamma-1)\ge \alpha\gamma-1$ for $0<\alpha<1$ and $0\le\gamma<1$,
we obtain
\[
\|A^\gamma(s^{\alpha-1}I-A)^{-1}\|
\le C_\gamma |s|^{\alpha\gamma-1},
\qquad 0<|s|\le1.
\]
By admissibility there exists $C_0>0$ such that $|\Kgen(s)|\le C_0$ for $|s|\le1$.
In addition, since $\Re(s)=|s|\cos\theta$ on $\Gtheta$ with $\cos\theta<0$,
we have $|e^{st}|\le e^{-c|s|t}$ for $c=-\cos\theta>0$.

\noindent As a consequence, for $0<|s|\le 1$, we obtain
\[
\|e^{st}\Kgen(s)A^\gamma(s^{\alpha-1}I-A)^{-1}\|
\le C\,|s|^{\alpha\gamma-1}e^{-c|s|t},
\]
for some constant $C>0$.

\noindent For $|s|\ge1$, admissibility gives $|\Kgen(s)|\le C_1|s|^{\alpha-1}$,
while the set $\{s^{\alpha-1}: s\in\Gtheta,\ |s|\ge1\}$ is compact and contained in $\Res(A)$.
Hence there exists $C_2>0$ such that
\[
\|A^\gamma(s^{\alpha-1}I-A)^{-1}\|\le C_2,
\qquad |s|\ge1.
\]
Thus, for $|s|\ge1$,
\[
\|e^{st}\Kgen(s)A^\gamma(s^{\alpha-1}I-A)^{-1}\|
\le C\,|s|^{\alpha-1}e^{-c|s|t}.
\]

\noindent  Combining the two regions yields an integrable majorant on $\Gtheta$.
Integrating along the rays and performing the change of variables $\rho=|s|t$, we obtain
\[
\|A^\gamma \Vgen_{\Kgen}(t)\|
\le C t^{-\alpha\gamma}
\int_0^\infty \rho^{\alpha\gamma-1}e^{-c\rho}\,d\rho
= C_\gamma\,t^{-\alpha\gamma},
\]
where $C_\gamma>0$ depends only on $\alpha$, $\gamma$ and the resolvent constants.
\end{proof}

% ----------------------------------------------------------
\subsection{Fractional smoothing}\label{subsec:smoothing}
% ----------------------------------------------------------

A central feature of analytic semigroup theory is \emph{smoothing}: even for rough initial data,
solutions become instantly more regular for $t>0$. In the almost sectorial setting one cannot
use classical semigroup arguments near $z=0$, but the contour mechanism still yields a sharp
fractional smoothing scale.

\begin{theorem}[Fractional smoothing]\label{thm:smoothing-abstract}
Let $\gamma\in[0,1)$ and assume \eqref{eq:fracpower-res}.
Then there exists $C_\gamma>0$ such that
\[
\|A^\gamma \Vgen_{\Kgen}(t)\|\le C_\gamma\,t^{-\alpha\gamma},\qquad t>0.
\]
\end{theorem}

\begin{proof}
Fix $t>0$. By definition of the $\Kgen$-resolvent family,
\[
A^\gamma \Vgen_{\Kgen}(t)
=
\frac{1}{2\pi i}
\int_{\Gtheta}
e^{st}\,\Kgen(s)\,A^\gamma(s^{\alpha-1}I-A)^{-1}\,ds.
\]
The insertion of $A^\gamma$ under the integral is justified since we will construct an absolutely integrable majorant in operator norm, and the fractional power $A^\gamma$ is a closed operator.

\medskip
Set $z=s^{\alpha-1}$.  
For $|s|\le 1$, Lemma~\ref{lem:geometry} ensures that $z\notin\Sigma_{\theta_A}$ and that $|z|=|s|^{\alpha-1}$ is large. Hence the high-frequency resolvent estimate~\eqref{eq:fracpower-res} applies and yields
\[
\|A^\gamma(zI-A)^{-1}\|
\le C_\gamma |z|^{\gamma-1}
= C_\gamma |s|^{(\alpha-1)(\gamma-1)}.
\]
Since $0<\alpha<1$ and $0\le\gamma<1$, we have $(1-\alpha)(1-\gamma)\ge0$ and thus
\[
(\alpha-1)(\gamma-1)
= \alpha\gamma-\alpha-\gamma+1
= \alpha\gamma-1+(1-\alpha)+(1-\gamma)
\ge \alpha\gamma-1.
\]
Because $|s|\le1$, this implies
\[
|s|^{(\alpha-1)(\gamma-1)}\le |s|^{\alpha\gamma-1}.
\]
Using admissibility of $\Kgen$, which gives $|\Kgen(s)|\le C_0$ for $|s|\le1$, and the exponential decay estimate $|e^{st}|\le e^{-c|s|t}$ on $\Gtheta$, the integrand is therefore bounded by
\[
\|e^{st}\Kgen(s)A^\gamma(s^{\alpha-1}I-A)^{-1}\|
\le C\,|s|^{\alpha\gamma-1}e^{-c|s|t},
\qquad |s|\le1.
\]

\medskip
For $|s|\ge1$, admissibility gives $|\Kgen(s)|\le C_\infty |s|^{\alpha-1}$.
Let
\[
E:=\{s^{\alpha-1}: s\in\Gtheta,\ |s|\ge1\}.
\]
Then $E$ is compact and contained in $\rho(A)$ by Lemma~\ref{lem:geometry}.
Since the map $z\mapsto A^\gamma(zI-A)^{-1}$ is continuous on $\rho(A)$, we obtain
\[
\sup_{z\in E}\|A^\gamma(zI-A)^{-1}\|<\infty.
\]
Consequently, for $|s|\ge1$,
\[
\|e^{st}\Kgen(s)A^\gamma(s^{\alpha-1}I-A)^{-1}\|
\le C\,|s|^{\alpha-1}e^{-c|s|t},
\]
which is integrable due to the exponential decay.

\medskip
Finally, parameterizing the contour rays as $s=re^{\pm i\theta}$, $r>0$, we have $|ds|=dr$ on each ray. Taking norms in the integral and combining the above estimates yields
\[
\|A^\gamma \Vgen_{\Kgen}(t)\|
\le C\int_0^\infty r^{\alpha\gamma-1}e^{-crt}\,dr.
\]
With the change of variables $\rho=rt$, this becomes
\[
\|A^\gamma \Vgen_{\Kgen}(t)\|
\le C\,t^{-\alpha\gamma}\int_0^\infty \rho^{\alpha\gamma-1}e^{-c\rho}\,d\rho
= C_\gamma\,t^{-\alpha\gamma},
\]
where $C_\gamma=C\,\Gamma(\alpha\gamma)/c^{\alpha\gamma}$. This concludes the proof.
\end{proof}

% ==========================================================
\section{Specialization to ABC and W dynamics}\label{sec:AB-W-specialization}
% ==========================================================

We now verify that both ABC and W kernels produce admissible multipliers, hence fit the abstract framework and generate well-posed dynamics for almost sectorial generators.

% ----------------------------------------------------------
\subsection{ABC (Atangana--Baleanu in Caputo sense)}\label{subsec:ABC}
% ----------------------------------------------------------

Let $0<\alpha<1$. The ABC derivative is characterized in the Laplace domain by
\begin{equation}\label{eq:Laplace-ABC}
\Lap\{{}^{\mathrm{ABC}}D_t^\alpha u\}(s)
=
\Phiab(s)\bigl(\widehat u(s)-s^{-1}u(0)\bigr),\qquad \Re s>0,
\end{equation}
with symbol
\begin{equation}\label{eq:Phi-ABC}
\Phiab(s)=B(\alpha)\,\frac{s^\alpha}{(1-\alpha)s^\alpha+\alpha}.
\end{equation}
The associated kernel multiplier is
\begin{equation}\label{eq:K-ABC}
\Kab(s):=\frac{\Phiab(s)}{s}
=
\frac{B(\alpha)}{1-\alpha}\,\frac{s^{\alpha-1}}{s^\alpha+c},
\qquad c=\frac{\alpha}{1-\alpha}.
\end{equation}

\begin{lemma}[ABC multiplier is admissible]\label{lem:ABC-admissible}
The function $\Kab$ is admissible in the sense of Definition~\ref{def:admissible-K}.
\end{lemma}

\begin{proof}
Analyticity holds on $\C\setminus(-\infty,0]$. We estimate $\Kab$ on the contour $\Gamma_\theta$.

\smallskip
\noindent\textbf{Case 1: $0<|s|\le 1$.}
Recall that $c=\alpha/(1-\alpha)>0$. Since the map $s\mapsto s^\alpha$ is continuous on $\Gamma_\theta$ and $|s^\alpha|\le 1$ when $|s|\le 1$, we have
\[
|s^\alpha+c|\ge c-|s^\alpha|\ge c-1.
\]
If $c\ge 2$ this yields $|s^\alpha+c|\ge c/2$. If $c<2$, then $s^\alpha+c$ still stays away from $0$ on $\Gamma_\theta$ (because $c>0$ and $\Gamma_\theta$ avoids the positive real axis), hence there exists $m_0>0$ such that
\[
|s^\alpha+c|\ge m_0
\qquad\text{for all } s\in\Gamma_\theta,\ 0<|s|\le 1.
\]
Therefore
\[
|\Kab(s)|
=\frac{B(\alpha)}{1-\alpha}\,\frac{|s|^{\alpha-1}}{|s^\alpha+c|}
\le \frac{B(\alpha)}{(1-\alpha)m_0}\,|s|^{\alpha-1}.
\]

\smallskip
\noindent\textbf{Case 2: $|s|\ge 1$.}
Then $|s^\alpha+c|\ge |s^\alpha|=|s|^\alpha$, and thus
\[
|\Kab(s)|
\le
\frac{B(\alpha)}{1-\alpha}\,\frac{|s|^{\alpha-1}}{|s|^\alpha}
=
\frac{B(\alpha)}{1-\alpha}\,|s|^{-1}
\le
\frac{B(\alpha)}{1-\alpha}\,|s|^{\alpha-1},
\]
because $\alpha-1\in(-1,0)$ and $|s|\ge 1$.

\smallskip
\noindent Combining the two cases gives \eqref{eq:K-bounds-abstract}, hence $\Kab$ is admissible.
\end{proof}

\noindent Define the ABC resolvent family by
\begin{equation}\label{eq:ABC-resolvent}
\Vgen_{\mathrm{ABC}}(t)
:=
\frac{1}{2\pi i}\int_{\Gtheta} e^{st}\,\Kab(s)\,(s^{\alpha-1}I-A)^{-1}\,ds,
\qquad t>0.
\end{equation}

\begin{theorem}[ABC evolution for almost sectorial $A$]\label{thm:ABC-evolution}
Let $A$ be almost sectorial. For every $u_0\in X$ and $f\in L^1_{\mathrm{loc}}([0,\infty);X)$, the problem
\[
{}^{\mathrm{ABC}}D_t^\alpha u(t)+Au(t)=f(t),\qquad u(0)=u_0,
\]
admits a unique mild solution
\[
u(t)=\Vgen_{\mathrm{ABC}}(t)u_0+\int_0^t \Vgen_{\mathrm{ABC}}(t-\tau)f(\tau)\,d\tau.
\]
Moreover, for every $\gamma\in[0,1)$,
\[
\|A^\gamma \Vgen_{\mathrm{ABC}}(t)\|\le C_\gamma\,t^{-\alpha\gamma},\qquad t>0.
\]
\end{theorem}

\begin{proof}
Apply Theorems~\ref{thm:strong-continuity}, \ref{thm:VoC-abstract}, and \ref{thm:smoothing-abstract} with $\Kgen=\Kab$.
\end{proof}

\medskip
\noindent\textbf{From ABC to W: an enriched admissible class.}
The W-operator extends the ABC kernel by introducing a second parameter $\beta\in(0,1]$.
This parameter changes the \emph{shape} of the memory while preserving the same resolvent geometry $z=s^{\alpha-1}$, and therefore the same smoothing scale.
The choice $\beta=1$ recovers the ABC multiplier, whereas $\beta<1$ produces a more spread-out memory effect.
At the level of Laplace multipliers, this yields a genuinely larger admissible family than the single ABC kernel.
From a numerical viewpoint, the factor $\bigl(1+(1-\alpha)s^{\alpha-1}\bigr)^{-\beta}$ provides a tunable damping of high-frequency contributions, which can be useful with rough data or irregular forcing.

\begin{remark}[On the regime $\beta>1$]\label{rem:beta>1}
The regime $\beta>1$ formally corresponds to a more concentrated memory kernel. In this case, the admissibility condition of Definition~\ref{def:admissible-K} may fail because the Laplace multiplier can exhibit excessive growth along the high-frequency part of the contour (depending on the chosen branch and sector).
A systematic analysis of $\beta>1$ is beyond the scope of this work and will be addressed separately.
\end{remark}

% ----------------------------------------------------------
\subsection{W dynamics (two-parameter kernel)}\label{subsec:W}
% ----------------------------------------------------------

Let $0<\alpha<1$ and $0<\beta\le1$. The W derivative is characterized by
\begin{equation}\label{eq:Laplace-W}
\Lap\{\WD_t^{\alpha,\beta}u\}(s)=
\Psym(s)\bigl(\widehat u(s)-s^{-1}u(0)\bigr),
\qquad \Re s>0,
\end{equation}
with symbol
\begin{equation}\label{eq:Psi-W}
\Psym(s)
=
\Bnorm\,\frac{s^\alpha}{\bigl(1+(1-\alpha)s^{\alpha-1}\bigr)^\beta}.
\end{equation}
The associated multiplier is
\begin{equation}\label{eq:K-W}
\Kw(s):=\frac{\Psym(s)}{s}
=
\Bnorm\,\frac{s^{\alpha-1}}{\bigl(1+(1-\alpha)s^{\alpha-1}\bigr)^\beta}.
\end{equation}

\begin{lemma}[W multiplier is admissible]\label{lem:W-admissible}
The function $\Kw$ is admissible in the sense of Definition~\ref{def:admissible-K}.
\end{lemma}

\begin{proof}
The function $s\mapsto \Kw(s)$ is analytic on $\C\setminus(-\infty,0]$ since $s^{\alpha-1}$ is analytic there (principal branch) and the denominator $\bigl(1+(1-\alpha)s^{\alpha-1}\bigr)^\beta$ never vanishes on $\Gtheta$ (with $\theta\in(\pi/2,\theta_A)$ fixed).

\smallskip
\noindent\textbf{Part 1: estimate for $0<|s|\le 1$.}
Write $s=re^{\pm i\theta}$ with $r\in(0,1]$. Then
\[
s^{\alpha-1}=r^{\alpha-1}e^{\pm i(\alpha-1)\theta},
\qquad\text{so that}\qquad |s^{\alpha-1}|=r^{\alpha-1}=|s|^{\alpha-1}.
\]
Since $\alpha\in(0,1)$ and $\theta\in(\pi/2,\theta_A)\subset(\pi/2,\pi)$, we have
$(\alpha-1)\theta\in(-\pi,0)$, hence $\arg(s^{\alpha-1})$ stays uniformly away from $0$ on the rays $\Gtheta^\pm$.
Consequently, there exists a constant $c_1>0$ (depending only on $\alpha$ and $\theta$) such that
\begin{equation}\label{eq:lower-bound-denom}
\bigl|1+(1-\alpha)s^{\alpha-1}\bigr|
\ge c_1\,|s^{\alpha-1}|,
\qquad 0<|s|\le 1,\ s\in\Gtheta.
\end{equation}
Raising \eqref{eq:lower-bound-denom} to the power $\beta\in(0,1]$ yields
\[
\bigl|1+(1-\alpha)s^{\alpha-1}\bigr|^\beta
\ge c_1^\beta\,|s^{\alpha-1}|^\beta.
\]
Therefore, for $0<|s|\le 1$ and $s\in\Gtheta$,
\[
|\Kw(s)|
=
\Bnorm\,\frac{|s|^{\alpha-1}}{\bigl|1+(1-\alpha)s^{\alpha-1}\bigr|^\beta}
\le
\Bnorm\,\frac{|s|^{\alpha-1}}{c_1^\beta\,|s^{\alpha-1}|^\beta}
=
\frac{\Bnorm}{c_1^\beta}\,|s|^{(1-\beta)(\alpha-1)}.
\]
Since $\alpha-1<0$ and $1-\beta\in[0,1)$, we have $(1-\beta)(\alpha-1)\le 0$,
and thus $|s|^{(1-\beta)(\alpha-1)}\le 1$ for $0<|s|\le 1$.
Hence
\[
|\Kw(s)|\le \frac{\Bnorm}{c_1^\beta},
\qquad 0<|s|\le 1,\ s\in\Gtheta,
\]
which provides the boundedness required in \eqref{eq:K-bounds-abstract} near the origin.

\smallskip
\noindent\textbf{Part 2: estimate for $|s|\ge 1$.}
Let $s\in\Gtheta$ with $|s|\ge 1$. Then $|s^{\alpha-1}|=|s|^{\alpha-1}\le 1$.
Hence
\[
\bigl|1+(1-\alpha)s^{\alpha-1}\bigr|
\ge 1-(1-\alpha)|s^{\alpha-1}|
\ge 1-(1-\alpha)=\alpha,
\]
so that
\[
\bigl|1+(1-\alpha)s^{\alpha-1}\bigr|^\beta \ge \alpha^\beta.
\]
It follows that, for all $s\in\Gtheta$ with $|s|\ge 1$,
\[
|\Kw(s)|
=
\Bnorm\,\frac{|s|^{\alpha-1}}{\bigl|1+(1-\alpha)s^{\alpha-1}\bigr|^\beta}
\le
\frac{\Bnorm}{\alpha^\beta}\,|s|^{\alpha-1}.
\]

\smallskip
\noindent\textbf{Conclusion.}
Combining the two regimes, we obtain constants
\[
C_0:=\frac{\Bnorm}{c_1^\beta},
\qquad
C_\infty:=\frac{\Bnorm}{\alpha^\beta},
\]
such that for all $s\in\Gtheta$,
\[
|\Kw(s)|\le
\begin{cases}
C_0, & 0<|s|\le 1,\\[1mm]
C_\infty |s|^{\alpha-1}, & |s|\ge 1.
\end{cases}
\]
This is exactly \eqref{eq:K-bounds-abstract}, hence $\Kw$ is admissible.
\end{proof}

\noindent Define the W resolvent family by
\begin{equation}\label{eq:W-resolvent}
\Vgen_{W}(t)
:=
\frac{1}{2\pi i}\int_{\Gtheta} e^{st}\,\Kw(s)\,(s^{\alpha-1}I-A)^{-1}\,ds,
\qquad t>0.
\end{equation}

\begin{theorem}[W evolution for almost sectorial $A$]\label{thm:W-evolution}
Let $A$ be almost sectorial. For every $u_0\in X$ and $f\in L^1_{\mathrm{loc}}([0,\infty);X)$, the problem
\[
\WD_t^{\alpha,\beta}u(t)+Au(t)=f(t),\qquad u(0)=u_0,
\]
admits a unique mild solution
\[
u(t)=\Vgen_{W}(t)u_0+\int_0^t \Vgen_{W}(t-\tau)f(\tau)\,d\tau.
\]
Moreover, for every $\gamma\in[0,1)$,
\[
\|A^\gamma \Vgen_{W}(t)\|\le C_\gamma\,t^{-\alpha\gamma},\qquad t>0.
\]
\end{theorem}

\begin{proof}
Apply the abstract results with $\Kgen=\Kw$.
\end{proof}

\begin{remark}
The regime $\beta>1$ formally corresponds to highly concentrated memory kernels. In this case, the admissibility condition on the Laplace multiplier may fail due to excessive growth at high frequencies. A systematic analysis of this regime is beyond the scope of the present work and will be addressed separately.
\end{remark}

% ==========================================================
\section{Comparison, scope, and scientific positioning}\label{sec:comparison-positioning}
% ==========================================================

\subsection{Caputo versus rational nonsingular kernels}

The obstruction in Section~\ref{sec:caputo-failure} is intrinsic: the Caputo symbol $s^\alpha$ forces resolvent evaluation near $z=0$. In contrast, rational nonsingular kernels that evaluate the resolvent at $s^{\alpha-1}$ bypass this difficulty through the same geometric redirection mechanism.

\subsection{Relation with earlier work on almost sectorial operators}

Abstract fractional Cauchy problems with almost sectorial operators were studied in several settings, notably in \cite{WangChenXiao2012}. Related functional analytic approaches to evolution equations under limited spectral control (including maximal regularity and functional calculus tools) can be found in \cite{KunstmannWeis2004,Haase2006}.
More recent resolvent-family viewpoints for fractional evolution equations are developed, for instance, in \cite{Bazhilov2020,LiZheng2015}. Most Caputo-based theories require additional assumptions near the origin. The present approach is orthogonal: it removes the need for low-frequency control by using a rational kernel structure that never requests resolvent bounds near $z=0$.

\medskip
 
\noindent Unlike the approach developed in \cite{WangChenXiao2012}, where the resolvent representation still involves Laplace inversion probing the operator near the origin, our framework is explicitly designed so that the inversion contour evaluates the resolvent only in the high-frequency regime $z = s^{\alpha-1} \to \infty$. This geometric redirection removes the need for additional low-frequency
assumptions on the generator and is the key mechanism allowing compatibility with almost sectorial operators.

\medskip
 
\noindent While the abstract resolvent-family approaches of
\cite{Bazhilov2020,LiZheng2015} provide a general functional analytic framework for fractional evolution equations, they do not address the geometric obstruction induced by Caputo-type Laplace symbols near $z=0$. Our contribution is orthogonal: we identify a class of admissible rational multipliers that enforce high-frequency resolvent evaluation, thereby restoring well-posedness for almost sectorial generators.

\subsection{Universality class}

Compatibility with almost sectorial generators is governed by two structural ingredients: (i) a kernel multiplier behaving as $s^{\alpha-1}$ at low frequencies, and (ii) high-frequency resolvent control of the generator. This yields a natural universality class of time-fractional models for degenerate diffusion.

% ==========================================================
\section{Applications to degenerate elliptic operators}\label{sec:applications}
% ==========================================================

The abstract framework is designed for generators whose resolvent is controlled only \emph{at high spectral frequencies}. This situation is typical for degenerate diffusion operators: they are often self-adjoint and dissipative in suitable weighted Hilbert
spaces, but the resolvent may be singular or poorly controlled near the origin.

\noindent A unifying feature is that, in the natural weighted $L^2$ spaces, the operators below are
self-adjoint and nonpositive, hence their spectrum lies in $(-\infty,0]$. Consequently,
the resolvent admits sectorial-type estimates outside any sector around the positive real axis,
uniformly for $|z|$ large. This matches Definition~\ref{def:almost-sectorial}.

% ----------------------------------------------------------
\subsection{Kimura-type diffusion operators}\label{subsec:Kimura}
% ----------------------------------------------------------

Consider the one-dimensional Kimura operator
\[
Ku(x)=x(1-x)u''(x),\qquad x\in(0,1).
\]
A natural Hilbert space is
\[
X=L^2\bigl((0,1),x^{-1}(1-x)^{-1}\,dx\bigr),
\]
Kimura-type operators and their degenerate diffusion structure are treated in detail in \cite{EpsteinMazzeo2013},  and $K$ can be realized as the self-adjoint nonpositive operator associated with the closed quadratic form
\[
\mathcal{E}(u,u)=\int_0^1 x(1-x)|u'(x)|^2\,dx.
\]

\begin{proposition}[Almost sectoriality of the Kimura operator]\label{prop:Kimura-AS}
For every $\theta\in(\pi/2,\pi)$ there exists $M_\theta>0$ such that
\[
\|(zI-K)^{-1}\|\le \frac{M_\theta}{|z|},\qquad z\notin\Sig_\theta,\quad |z|\ge 1.
\]
In particular, $K$ is almost sectorial.
\end{proposition}

\begin{proof}
Since $K$ is self-adjoint and nonpositive, $\sigma(K)\subset(-\infty,0]$ and $\|(zI-K)^{-1}\|=\dist(z,\sigma(K))^{-1}=\dist(z,(-\infty,0])^{-1}$. If $z\notin\Sig_\theta$, geometry yields $\dist(z,(-\infty,0])\ge c_\theta |z|$, hence the bound follows.
\end{proof}

\begin{theorem}[Kimura evolution: well-posedness and smoothing]\label{thm:Kimura-evolution}
Let $\alpha\in(0,1)$ and $\beta\in(0,1]$. For every $u_0\in X$ and $f\in L^1_{\mathrm{loc}}([0,\infty);X)$, both problems
\[
{}^{\mathrm{ABC}}D_t^\alpha u(t)+Ku(t)=f(t),\quad u(0)=u_0,
\qquad\text{and}\qquad
\WD_t^{\alpha,\beta}u(t)+Ku(t)=f(t),\quad u(0)=u_0,
\]
admit unique mild solutions given by the corresponding resolvent families. Moreover, for every $\gamma\in[0,1)$,
\[
\|K^\gamma u(t)\|\le C\,t^{-\alpha\gamma}\Bigl(\|u_0\|+\int_0^t \|f(\tau)\|\,d\tau\Bigr).
\]
\end{theorem}

\begin{proof}
Combine Proposition~\ref{prop:Kimura-AS} with Theorems~\ref{thm:ABC-evolution} and \ref{thm:W-evolution}, and use the smoothing estimate from Theorem~\ref{thm:smoothing-abstract}.
\end{proof}

% ----------------------------------------------------------
\subsection{Bessel-type radial diffusion operators}\label{subsec:Bessel}
% ----------------------------------------------------------

Let $\nu>-1/2$ and consider
\[
\mathcal{B}_\nu u(r)=u''(r)+\frac{2\nu+1}{r}u'(r),\qquad r>0,
\]
in the weighted space
\[
X_\nu=L^2\bigl((0,\infty),r^{2\nu+1}\,dr\bigr).
\]
Bessel-type operators, heat kernel bounds, and related spectral estimates are classically discussed in \cite{Davies1989}. The operator $\mathcal{B}_\nu$ can be realized as self-adjoint and nonpositive.

\begin{proposition}[Almost sectoriality of the Bessel operator]\label{prop:Bessel-AS}
For every $\theta\in(\pi/2,\pi)$ there exists $M_\theta>0$ such that
\[
\|(zI-\mathcal{B}_\nu)^{-1}\|\le \frac{M_\theta}{|z|},\qquad z\notin\Sig_\theta,\quad |z|\ge 1.
\]
\end{proposition}

\begin{proof}
Same argument as Proposition~\ref{prop:Kimura-AS}, using $\sigma(\mathcal{B}_\nu)\subset(-\infty,0]$.
\end{proof}

\begin{theorem}[Bessel evolution]\label{thm:Bessel-evolution}
Let $\alpha\in(0,1)$ and $\beta\in(0,1]$. For every $u_0\in X_\nu$ and $f\in L^1_{\mathrm{loc}}([0,\infty);X_\nu)$, the corresponding ABC and W problems admit unique mild solutions, and for every $\gamma\in[0,1)$,
\[
\|\mathcal{B}_\nu^\gamma u(t)\|_{X_\nu}\le
C\,t^{-\alpha\gamma}\Bigl(\|u_0\|_{X_\nu}+\int_0^t\|f(\tau)\|_{X_\nu}\,d\tau\Bigr).
\]
\end{theorem}

\begin{proof}
Combine Proposition~\ref{prop:Bessel-AS} with the abstract theory and the smoothing estimate.
\end{proof}

% ==========================================================
\section{Mild and weak solutions}\label{sec:mild-weak}
% ==========================================================

We state a convenient mild/weak equivalence in the present setting.

\subsection{Mild solutions}

\begin{definition}[Mild solution]\label{def:mild}
Let $u_0\in X$ and $f\in L^1_{\mathrm{loc}}([0,\infty);X)$. A function $u:[0,\infty)\to X$
is a \emph{mild solution} of
\begin{equation}\label{eq:abstract-eq}
\mathcal{D}u(t)+Au(t)=f(t),\qquad u(0)=u_0,
\end{equation}
(where $\mathcal{D}$ denotes either ${}^{\mathrm{ABC}}D_t^\alpha$ or $\WD_t^{\alpha,\beta}$) if it satisfies the variation-of-constants formula
\[
u(t)=\Vgen(t)u_0+\int_0^t \Vgen(t-\tau)\,f(\tau)\,d\tau,
\]
with $\Vgen$ the corresponding resolvent family.
\end{definition}

\subsection{Weak solutions}

Let $A^*$ denote the adjoint of $A$ (in Banach spaces, use the appropriate dual pairing).

\begin{definition}[Weak solution]\label{def:weak}
A function $u\in L^1_{\mathrm{loc}}([0,\infty);X)$ is a \emph{weak solution} of \eqref{eq:abstract-eq}
if for every $\varphi\in \Dom(A^*)$ and every $t>0$,
\[
\int_0^t \langle u(\tau),\,\mathcal{D}_{t-\tau}\varphi\rangle\,d\tau
+\int_0^t \langle Au(\tau),\varphi\rangle\,d\tau
=
\int_0^t \langle f(\tau),\varphi\rangle\,d\tau
+\langle u_0,\varphi\rangle,
\]
where $\langle\cdot,\cdot\rangle$ denotes the dual pairing.
\end{definition}

\subsection{Equivalence}

\begin{theorem}[Equivalence of mild and weak solutions]\label{thm:mild-weak-equivalence} Let $A$ be almost sectorial, $u_0\in X$, and $f\in L^1_{\mathrm{loc}}([0,\infty);X)$.
Then $u$ is a mild solution if and only if it is a weak solution in the sense of Definition~\ref{def:weak}.
\end{theorem}

\begin{proof}
We prove both implications and emphasize the uniqueness principle for Laplace transforms in $L^1_{\mathrm{loc}}$.

\medskip
\noindent\textbf{(i) Mild $\Rightarrow$ weak.}
Assume $u$ is a mild solution in the sense of Definition~\ref{def:mild}, i.e.
\[
u(t)=\Vgen(t)u_0+\int_0^t \Vgen(t-\tau)\,f(\tau)\,d\tau,\qquad t\ge 0,
\]
where $\Vgen$ is the corresponding resolvent family (ABC or W).
Fix $\varphi\in\Dom(A^*)$ and take the dual pairing with $\varphi$.
Using Fubini's theorem (justified by the $L^1_{\mathrm{loc}}$ assumption on $f$ and the boundedness of $\Vgen(t)$ on compact time intervals), one obtains the weak identity in Definition~\ref{def:weak}. Equivalently, one may take Laplace transforms of both sides: Lemma~\ref{lem:Laplace-VK} gives, for $\Re s>0$,
\[
\widehat{u}(s)=\Kgen(s)\,(s^{\alpha-1}I-A)^{-1}\,u_0
+\Kgen(s)\,(s^{\alpha-1}I-A)^{-1}\,\widehat{f}(s),
\]
and testing against $\varphi$ yields the Laplace-domain form of the weak formulation, which inverts back to Definition~\ref{def:weak}.

\medskip
\noindent\textbf{(ii) Weak $\Rightarrow$ mild.}
Assume $u\in L^1_{\mathrm{loc}}([0,\infty);X)$ satisfies the weak formulation (Definition~\ref{def:weak}). Fix $\varphi\in\Dom(A^*)$. Since $t\mapsto \langle u(t),\varphi\rangle$ and $t\mapsto \langle f(t),\varphi\rangle$ belong to $L^1_{\mathrm{loc}}([0,\infty))$, their Laplace transforms exist for $\Re s>0$. Taking Laplace transforms in the weak identity and using the Laplace characterizations \eqref{eq:Laplace-ABC}--\eqref{eq:K-ABC} (ABC case) or \eqref{eq:Laplace-W}--\eqref{eq:K-W} (W case), we obtain, for $\Re s>0$,
\[
\langle \widehat{u}(s),\varphi\rangle
=
\Big\langle \Kgen(s)\,(s^{\alpha-1}I-A)^{-1}\big(u_0+\widehat{f}(s)\big),\,\varphi\Big\rangle,
\]
hence
\[
\widehat{u}(s)
=
\Kgen(s)\,(s^{\alpha-1}I-A)^{-1}\big(u_0+\widehat{f}(s)\big)
\quad\text{in }X,\qquad \Re s>0.
\]
By Lemma~\ref{lem:Laplace-VK}, the right-hand side is the Laplace transform of
\[
t\mapsto \Vgen(t)u_0+\int_0^t \Vgen(t-\tau)\,f(\tau)\,d\tau,
\]
which is therefore an $L^1_{\mathrm{loc}}([0,\infty);X)$-function.
Finally, by the \emph{uniqueness of Laplace transforms in $L^1_{\mathrm{loc}}([0,\infty);X)$} (see, e.g., \cite[Theorem.~1.7.3]{ArendtBattyHieberNeubrander2001}), we conclude that $u$ coincides a.e.\ with this variation-of-constants expression, and by continuity of the mild solution it holds for all $t\ge 0$. Thus $u$ is a mild solution.

\noindent Uniqueness follows from the uniqueness of Laplace transforms in $L^1_{\mathrm{loc}}([0,\infty);X)$. This proves the equivalence.
\end{proof}

% =========================================================
\section{Numerical illustrations}
\label{sec:numerics}
% =========================================================

This section provides numerical illustrations of the fractional smoothing results established in Theorem~\ref{thm:smoothing-abstract}. The purpose is not to design a convergent numerical scheme nor to investigate sharp asymptotic constants, but rather to validate the \emph{qualitative behavior} predicted by the abstract theory for representative almost sectorial generators. For background on numerical and modeling aspects of fractional diffusion and related nonlocal-in-time dynamics, see, e.g., \cite{MeerschaertScheffler2004,Lischke2020}.

\noindent Two complementary situations are considered:
\begin{itemize}
\item a \emph{bounded degenerate operator} (Kimura type) combined with the Atangana--Baleanu kernel ($\beta=1$);
\item a \emph{singular unbounded operator} (Bessel type) combined with the W-kernel ($0<\beta<1$).
\end{itemize}
This allows us to compare ABC and W dynamics within the same resolvent framework and to highlight the role of the additional parameter $\beta$.

% ---------------------------------------------------------
\subsection{Kimura operator: ABC dynamics ($\beta=1$)}
\label{subsec:numerics-kimura}
% ---------------------------------------------------------

We first consider the Kimura diffusion operator $K$ on $(0,1)$, realized as in Section~\ref{sec:applications}. We fix the parameters $\alpha=0.5$ and $\gamma=0.5$. In this case, the abstract fractional smoothing estimate Theorem~\ref{thm:smoothing-abstract} predicts the uniform bound
\[
\|K^\gamma \Vgen_{\mathrm{ABC}}(t)\|
\le C\,t^{-\alpha\gamma},
\qquad t>0.
\]

\noindent To construct a discrete proxy for $\|K^\gamma \Vgen_{\mathrm{ABC}}(t)u_0\|$, we take $u_0=\sin(\pi x)$ as an initial datum, wich belongs to the weighted space $X=L^2((0,1),x^{-1}(1-x)^{-1}\,dx)$ and sufficiently smooth. The operator $K$ is approximated by a piecewise linear (P1) finite element discretization on a uniform mesh with $N = 1000$ interior points of the associated Dirichlet form on a uniform mesh, yielding a negative semidefinite matrix $K_h$. The discrete evolution is defined through a numerical approximation of the Laplace inversion formula
\[
u_h(t)
=
\frac{1}{2\pi i}
\int_{\Gamma_\theta}
e^{st}\,K_\alpha(s)\,(s^{\alpha-1}I-K_h)^{-1}u_{0,h}\,ds,
\]
where the contour integral is evaluated along the rays of $\Gamma_\theta$. We then set $w_h(t)=K_h^\gamma u_h(t)$, where $K_h^\gamma$ is defined by the spectral calculus for symmetric matrices.

\noindent Figure~\ref{fig:kimura-decay} displays the decay of $\|w_h(t)\|$ on a log--log scale, together with the reference slopes $t^{-\alpha\gamma}$ and $t^{-\gamma}$.
The numerical curve remains uniformly below the bound $t^{-\alpha\gamma}$, in agreement with Theorem~\ref{thm:smoothing-abstract}. On intermediate and large time scales, the decay becomes faster and approaches the spatial rate $t^{-\gamma}$, reflecting the dominance of low-frequency spectral modes of the Kimura operator. This behavior is consistent with the fact that Theorem~\ref{thm:smoothing-abstract} provides a global upper bound rather than an asymptotic equivalence.

\begin{figure}[H]
\centering
\includegraphics[width=0.75\textwidth]{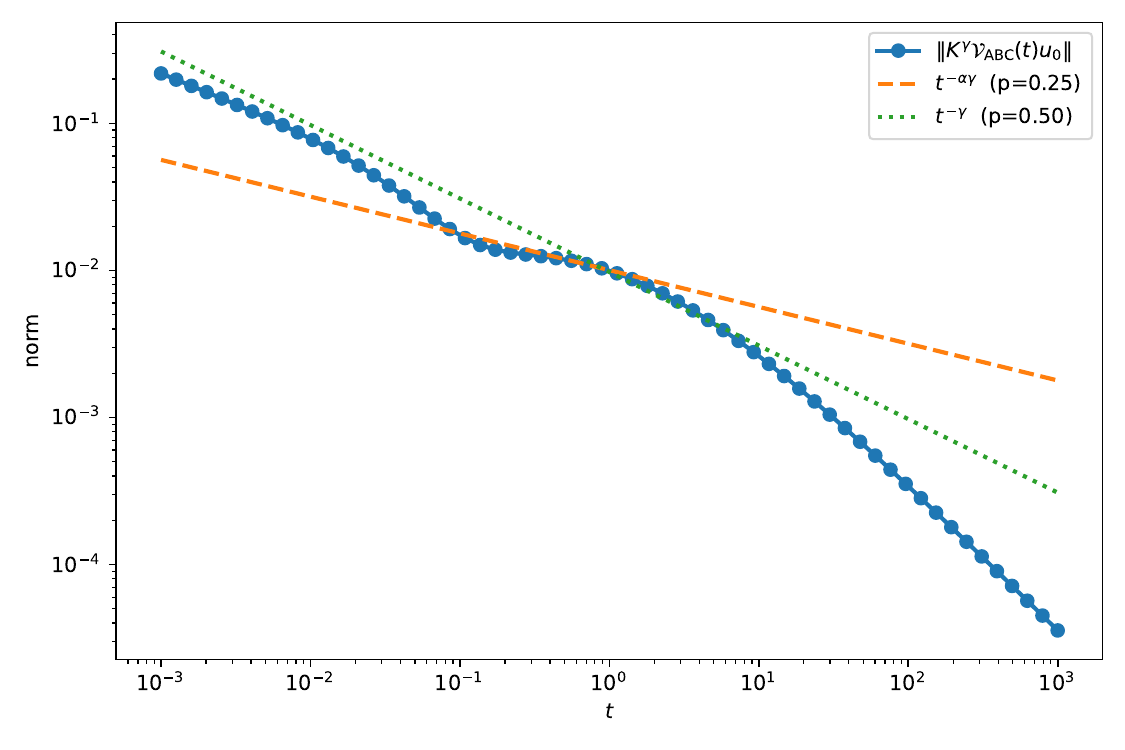}
\caption{Decay of $\|K^\gamma \Vgen_{\mathrm{ABC}}(t)u_0\|$ for the Kimura operator with
ABC dynamics ($\alpha=0.5$, $\gamma=0.5$).
The dashed line corresponds to the theoretical bound $t^{-\alpha\gamma}$,
while the dotted line shows the faster reference rate $t^{-\gamma}$.}
\label{fig:kimura-decay}
\end{figure}

\noindent To further clarify this transition, Figure~\ref{fig:kimura-slope} reports the local decay exponent
\[
-\frac{d\log \|w_h(t)\|}{d\log t}
\]
as a function of time. For small times, the exponent is close to $\alpha\gamma$, indicating a genuinely fractional regime.
For larger times, it gradually approaches $\gamma$, corresponding to a spatially dominated decay. This transition does not contradict the abstract estimate, but rather illustrates its non-asymptotic character.

\begin{figure}[H]
\centering
\includegraphics[width=0.75\textwidth]{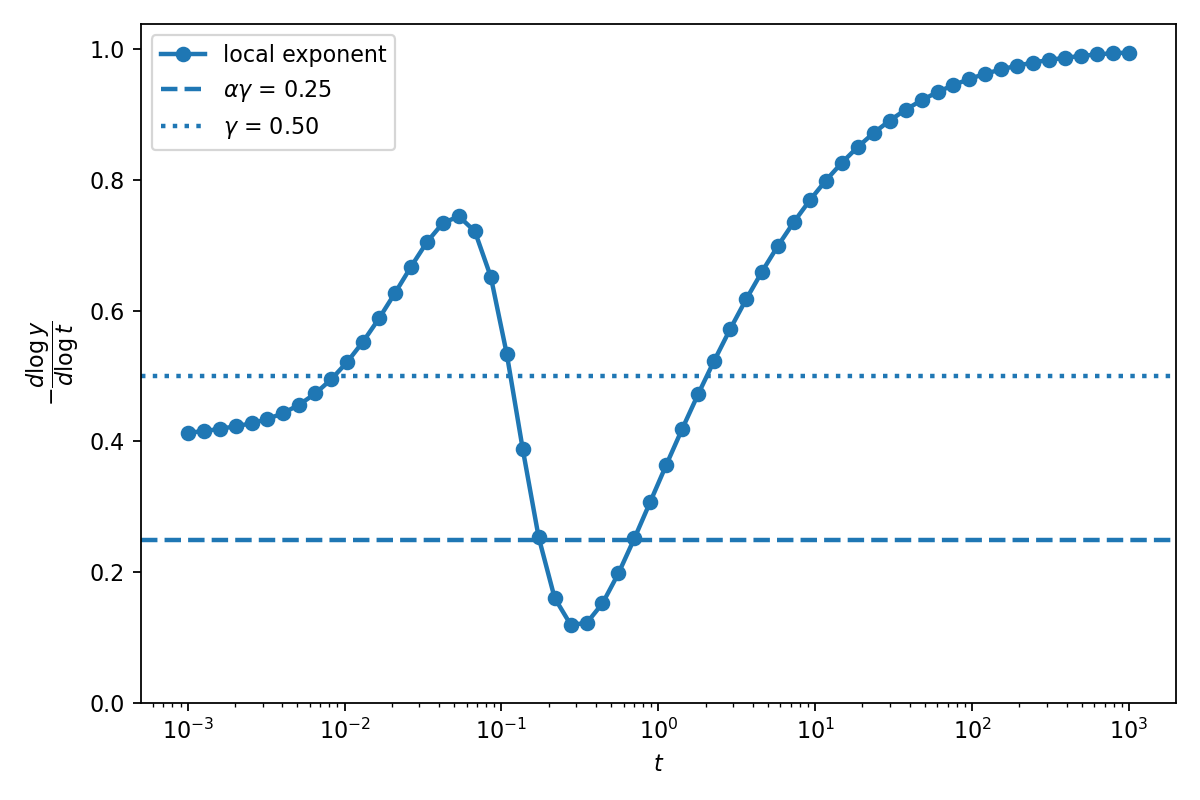}
\caption{Local decay exponent
$-\frac{d\log \|K^\gamma \Vgen_{\mathrm{ABC}}(t)u_0\|}{d\log t}$ for the Kimura operator.
The horizontal dashed and dotted lines correspond to $\alpha\gamma$ and
$\gamma$, respectively.}
\label{fig:kimura-slope}
\end{figure}

% ---------------------------------------------------------
\subsection{Bessel operator: W dynamics ($0<\beta<1$)}
\label{subsec:numerics-bessel}
% ---------------------------------------------------------

We now turn to the Bessel-type operator $\mathcal{B}_\nu$ acting on the weighted
space $X_\nu=L^2((0,\infty),r^{2\nu+1}\,dr)$, as introduced in
Section~\ref{sec:applications}.
In contrast with the Kimura case, $\mathcal{B}_\nu$ is unbounded and singular at the origin.
This makes it a natural test case for the W-operator, which introduces an additional parameter $\beta$ at the level of the Laplace multiplier.

\noindent We fix $\nu=0.25$, $\alpha=0.5$, $\beta=0.8$, and $\gamma=0.5$.
The abstract theory predicts the same fractional smoothing bound
\[
\|\mathcal{B}_\nu^\gamma \Vgen_W(t)\|
\lesssim
t^{-\alpha\gamma},
\qquad t>0,
\]
independently of $\beta$.

\noindent The numerical procedure parallels the Kimura case.
The operator $\mathcal{B}_\nu$ is discretized by a symmetric scheme adapted to the weighted space, yielding a negative semidefinite matrix $B_h$. The W-resolvent family is approximated via contour integration using the multiplier $\mathcal{K}_W(s)$, and we set
$w_h(t)=B_h^\gamma u_h(t)$.

\noindent Figure~\ref{fig:numerics-bessel-decay} shows the decay of
$\|\mathcal{B}_\nu^\gamma u(t)\|$ on a log--log scale. As in the Kimura--ABC case, the numerical curve remains uniformly below the reference rate $t^{-\alpha\gamma}$. The plots also indicate that the transient regime is visibly affected by the choice $\beta<1$: the decay is initially slower, reflecting a more spread-out memory kernel. At larger times, the decay accelerates and approaches the spatial rate $t^{-\gamma}$, again due to spectral effects.

\begin{figure}[H]
\centering
\includegraphics[width=0.75\textwidth]{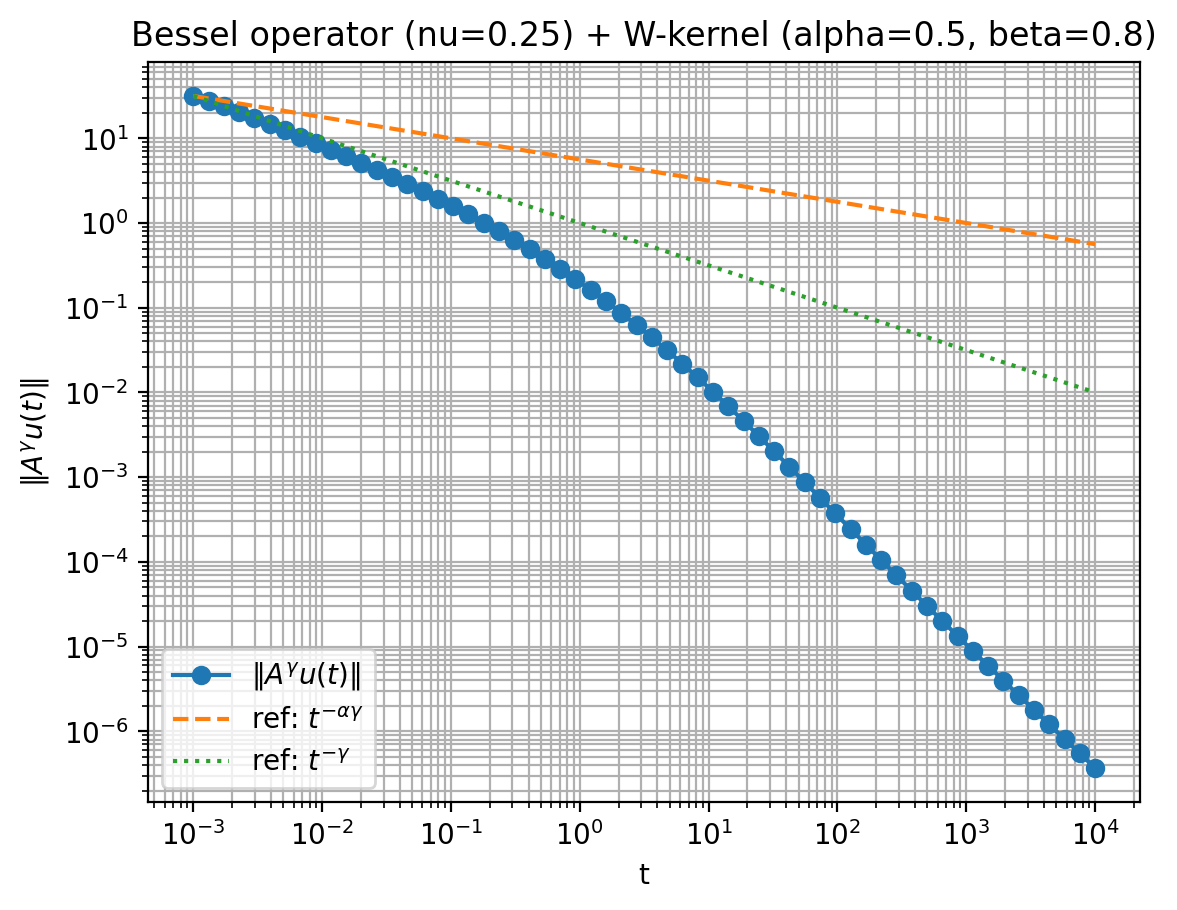}
\caption{Decay of $\|\mathcal{B}_\nu^\gamma u(t)\|$ for the Bessel operator ($\nu=0.25$) with W dynamics ($\alpha=0.5$, $\beta=0.8$, $\gamma=0.5$). The dashed and dotted lines correspond to the reference rates $t^{-\alpha\gamma}$ and $t^{-\gamma}$, respectively.} \label{fig:numerics-bessel-decay}
\end{figure}

\noindent The corresponding local decay exponent is shown in
Figure~\ref{fig:numerics-bessel-slope}. In agreement with the abstract smoothing bound, the exponent starts near $\alpha\gamma$ at short times and gradually transitions toward $\gamma$ at
larger times. Compared with the ABC case, this transition is smoother and more progressive, highlighting the additional modeling flexibility introduced by the parameter $\beta$, while preserving the same resolvent geometry and the same global smoothing scale.

\begin{figure}[H]
\centering
\includegraphics[width=0.7\textwidth]{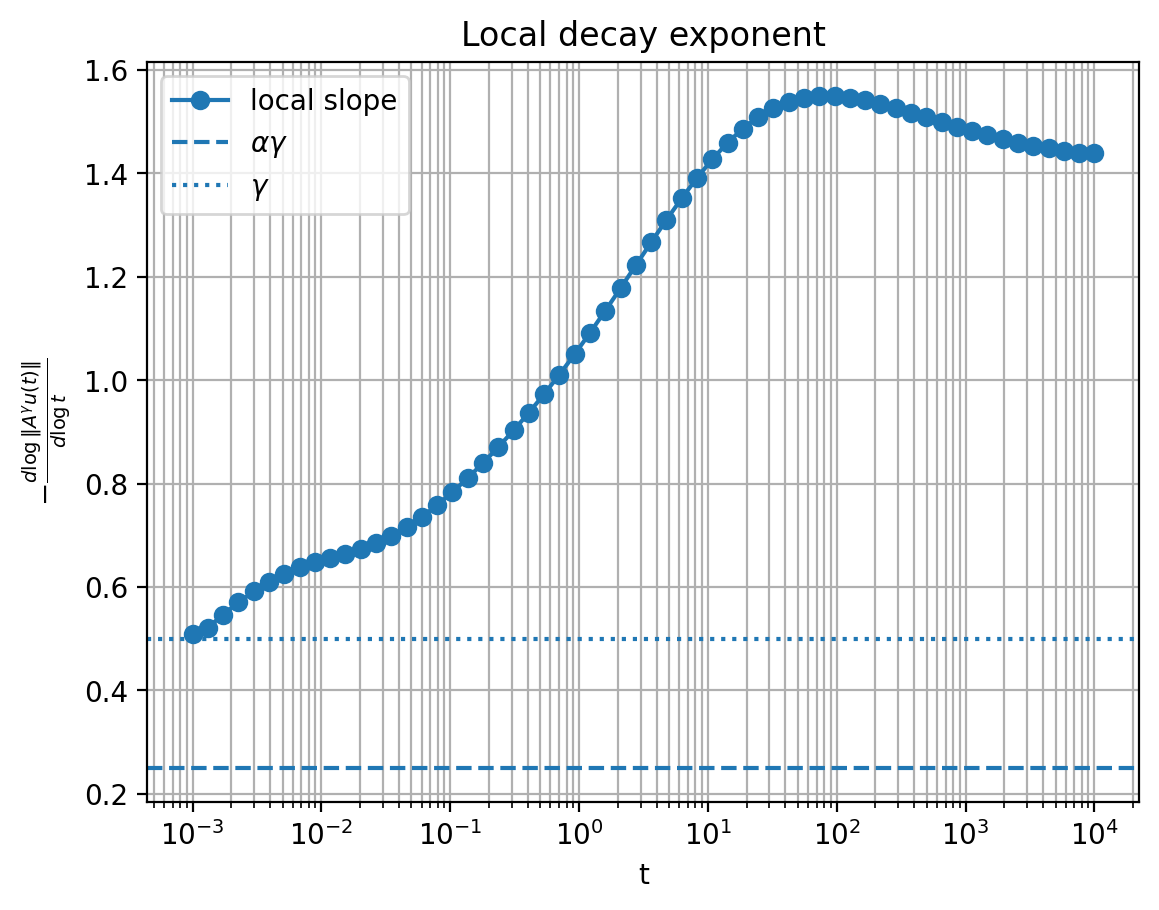}
\caption{Local decay exponent
$-\frac{d\log \|\mathcal{B}_\nu^\gamma u(t)\|}{d\log t}$
for the Bessel operator with W dynamics. The horizontal reference lines correspond to $\alpha\gamma$ and $\gamma$.}
\label{fig:numerics-bessel-slope}
\end{figure}

\medskip
These numerical experiments are consistent with the abstract smoothing estimate for both ABC and W kernels. They illustrate that the W-operator preserves the same fractional smoothing scale while providing additional flexibility in the transient regime, a feature that is invisible at the level of resolvent geometry but clearly observable in time-dependent behavior.

\paragraph{Comparison: Kimura (ABC) vs.\ Bessel (W).}
Comparing the Kimura--ABC and Bessel--W experiments helps disentangle what comes from the kernel and what comes from the spatial operator. In both settings, the early-time decay of $\|A^\gamma u(t)\|$ follows the fractional scale
$t^{-\alpha\gamma}$ predicted by Theorem~\ref{thm:smoothing-abstract}, which is consistent with the same geometric redirection mechanism underlying the contour representation. At larger times, the observed decay departs from this reference slope and progressively reflects the spectral distribution of the underlying operator, leading to a faster, spatially dominated regime.

\noindent The additional parameter $\beta$ in the W-kernel does not change the \emph{predicted} smoothing scale, but it visibly affects the transient: for $\beta<1$ the crossover between the short-time
fractional regime and the long-time spatial regime is typically smoother and more gradual than in the ABC case (which corresponds to $\beta=1$). Overall, these plots should be read as qualitative consistency checks: they support the fact that the short-time bound is governed by the resolvent geometry, while the detailed crossover dynamics depends on both the kernel shape (through $\beta$) and the spectrum of $A$.

% ==========================================================
\section{Conclusion}\label{sec:conclusion}
% ==========================================================

We developed a unified resolvent framework for fractional evolution equations driven by rational nonsingular kernels compatible with almost sectorial operators. The key mechanism is the Laplace-domain mapping $z=s^{\alpha-1}$, which redirects low Laplace frequencies to high spectral parameters. This geometric argument explains the incompatibility of Caputo dynamics in the almost sectorial regime and provides a systematic treatment of both ABC and W derivatives, including well-posedness, a variation-of-constants formula,
and sharp fractional smoothing estimates, with applications to degenerate diffusion generators. Unlike Caputo-based approaches for almost sectorial generators, which typically require additional spectral assumptions near the origin, the present framework identifies a kernel-driven universality class in which the geometry of Laplace inversion alone guarantees compatibility; see also \cite{WangChenXiao2012,Bazhilov2020}.

% ==========================================================

\end{document}